\documentclass{amsart}
\usepackage{amsmath,amssymb,latexsym, amsfonts, amscd, amsthm}

\usepackage{hyperref}
\hypersetup{colorlinks=false}
 \usepackage[foot]{amsaddr}

\usepackage{verbatim}
\usepackage{xcolor}
\hypersetup{colorlinks=false,linkbordercolor=red,linkcolor=green,pdfborderstyle={/S/U/W 1}}

\usepackage{bbm}
\usepackage{enumerate}
\usepackage[all]{xy}

\newtheorem{thm}{Theorem}[section]

\newtheorem{cor}[thm]{Corollary}

\theoremstyle{definition}

\theoremstyle{remark}
\newtheorem{rem}[thm]{Remark}

\numberwithin{equation}{section}

\DeclareMathOperator*{\End}{End}
\DeclareMathOperator*{\Irr}{Irr}
\DeclareMathOperator*{\disc}{disc}
\DeclareMathOperator*{\Gal}{Gal}
\DeclareMathOperator*{\Res}{Res}

\DeclareMathOperator*{\ad}{ad}

\DeclareMathOperator*{\der}{der}
\DeclareMathOperator*{\Hom}{Hom}
\DeclareMathOperator*{\J}{JH}

\DeclareMathOperator*{\SU}{SU}
\DeclareMathOperator*{\GU}{GU}
\DeclareMathOperator*{\U}{U}
\DeclareMathOperator*{\SL}{SL}
\DeclareMathOperator*{\GL}{GL}
\DeclareMathOperator*{\GSp}{GSp}
\DeclareMathOperator*{\Spin}{Spin}
\DeclareMathOperator*{\GSpin}{GSpin}

\newcommand{\s}{\simeq}

\newcommand{\CC}{\mathbb{C}}
\newcommand{\QQ}{\mathbb{Q}}
\newcommand{\ZZ}{\mathbb{Z}}

\def\bA{\bold A}
\def\bG{\bold G}

\def\bP{\bold P}
\def\bM{\bold M}
\def\bN{\bold N}
\def\bB{\bold B}
\def\bT{\bold T}

\begin{document}

\title[Weakly unramified representations and $R$-groups]{Weakly unramified representations, finite morphisms, and Knapp-Stein $R$-groups}

\author[Kwangho Choiy]{Kwangho Choiy} 
\address{School of Mathematical and Statistical Sciences,
Southern Illinois University,
Carbondale, IL 62901-4408,
U.S.A.}
\email{kchoiy@siu.edu}

\keywords{weakly unramified representations, Kottwitz homomorphism, non-quasi-split algebraic group, Knapp-Stein $R$-group, restriction}

\subjclass[2010]{Primary \textbf{11F70}; Secondary 11E95, 20G20, 22E50}

\begin{abstract}
We transfer Knapp-Stein $R$-groups for unitary weakly unramified characters between a $p$-adic quasi-split group and its non-quasi-split inner forms, and provide the structure of those $R$-groups for a general connected reductive group over a $p$-adic field. This work supports previous studies on the behavior of $R$-groups between inner forms, and extends Keys' classification for unitary unramified cases of simply-connected, almost simple, semi-simple groups.
\end{abstract}

\maketitle
\section{Introduction} \label{intro}

Unramified representations of $p$-adic groups are one of essential classes in representation theory \cite{bdk, bk93, bk98, keys82unram} and further have geometric applications \cite{haines14, hr21, hai10, shin12}.
Focused on their representation-theoretic aspects, we study the reducibility of principal series representations of $p$-adic groups induced from unitary weakly unramified characters, whose category is larger than that of unitary unramified characters.

The purpose of this paper is to transfer Knapp-Stein $R$-groups for unitary weakly unramified characters between a $p$-adic quasi-split group and its non-quasi-split inner forms; and to provide the general structure of those $R$-groups for principal series representations of an arbitrary reductive algebraic group induced from unitary weakly unramified characters.
The study of the transfer between inner forms has been conducted in several cases, including $\SL_n$ and its inner forms \cite{chaoli, chgo12} and classical groups \cite{cgclassic, cgsu, han04}. Our result on those $R$-groups extends Keys' classification \cite{keys82unram} of the $R$-groups for unitary unramified characters of minimal Levi subgroups of simply-connected groups over a $p$-adic field to unitary weakly unramified cases of general connected reductive groups whose derived groups are simply-connected, almost-simple.
In stages, we generalize the $R$-groups for unitary unramified characters from the simply-connected groups to such arbitrary reductive algebraic groups, and extend the $R$-group for unitary unramified characters to unitary weakly unramified cases.
We also describe the difference between those $R$-groups in each direction. 

To be precise, throughout the introduction, we let $\bG$ be a connected reductive algebraic group over a $p$-adic field $F$ of characteristic 0 and let $\bM$ be a minimal $F$-Levi subgroup of $\bG.$ Denote by $\bar{F}$ an algebraic closure of $F$ and by $\Gamma$ the absolute Galois group $\Gal(\bar{F} / F).$ We write $G, M$ for the group $\bG(F), \bM(F)$ of their $F$-rational points, respectively, and this notation shall be applied to other algebraic groups.

We consider the group $X^{\rm{w}} (M)$ of weakly unramified characters of $M.$ It is noted that the connected component of $X^{\rm{w}} (M)$ is the subgroup of unramified characters of $M$ and two groups are identical in certain settings, but not in general (see Section \ref{subsub wurep}).  
For the nature of $R$-groups, we specialize in discrete series representations of $F$-Levi subgroups so that those characters in our question are unitary. Given a unitary character $\chi$ in $X^{\rm{w}} (M),$ the group structure of the Knapp-Stein $R$-group $R_\chi$ is of great interest. In particular, the group $R_\chi$ governs the $\CC$-algebra structure of the commuting algebra ${\End}_{G}(i_M^G (\chi))$ of $G$-endomorphisms of the parabolic induction $i_M^G (\chi)$ and then its reducibility (see Section \ref{notation} and Remark \ref{rem for commuting alg}). Also, the $R$-group  determines the number of conjugacy classes of good maximal compact (special maximal parahoric) subgroups of $G$ (see \cite[Section 3]{keys82unram} and Remark \ref{rem for K}). 

Revisiting Keys' classification of $R_\chi$ in \cite{keys82unram} for a simply-connected, almost simple, semi-simple group $\bG$ and a unitary unramified character $\chi$ of $M$ (see Section \ref{unram rgp-keys}), we first extend his classification to non-simply-connected cases. Our generalization is done via the setting of a connected reductive group $\bG$ such that its derived subgroup $\bG_{\der}$ is simply-connected and almost simple. To this end, we move to the intersection $\bM \cap \bG_{\der},$ denoted by $\bM_\flat,$ which is a minimal $F$-Levi subgroup of $\bG_{\der}.$ 
Given a unitary unramified character $\chi$ of $M,$ we obtain a unitary unramified character $\chi_{\flat}$ of $M_{\flat}$ which is an irreducible constituent in the restriction  ${\Res}_{M_\flat}^{M}(\chi)$ of $\chi$ from $M$ to $M_\flat.$ It then follows that 
\[
\chi_\flat = {\Res}_{M_\flat}^{M}(\chi)
\] 
due to \cite[Lemma 2.1(d)]{gk82} and \cite[Proposition 2.4]{tad92} (see Section \ref{gen to non-sc} for further details).
We employ a finite abelian group $\widehat{W(\chi)}$ consisting of characters $\eta$ on $M/M_\flat$ such that ${^w}\chi \s \chi \eta$ for some Weyl group element $w$ stabilizing $\chi_\flat$ (see \eqref{hatW}). We then have the following exact sequence (Theorem \ref{thm for general}):
\begin{equation} \label{intro exseq 1}
	1 \longrightarrow R_{\chi}  \longrightarrow R_{\chi_\flat} \longrightarrow \widehat{W(\chi)} 
	\longrightarrow 1.
\end{equation}
It turns out that this extends the Keys' work for unitary unramified cases from simply-connected groups to such general groups.

Now, we transfer the Knapp-Stein $R$-groups for unitary weakly unramified characters between inner forms. For this, we employ a quasi-split inner form of $\bG.$ Namely, we consider a connected reductive $F$-quasi-split group $\bG^*$ such that $(\bG, \Psi)$ is an $F$-inner form of $\bG^*$ with a $\Gamma$-stable $\bG^*_{\ad}(\bar F)$-orbit $\Psi$ of $\bar F$-isomorphisms $\psi: \bG \rightarrow \bG^*.$ Here, $\bar F$ is an algebraic closure of $F;$  $\Gamma$ is the absolute Galois group over $F;$ and $\bG^*_{\ad}$ denotes the adjoint group of $\bG^*.$
Following \cite[Section 11]{haines14} and \cite[Section 8]{haines15}, we obtain an $F$-Levi subgroup $\bM^*$ of $\bG^*$ such that $\psi_0(\bM)=\bM^*$ for some $\psi_0 \in \Psi.$ 
Denote by $\bA$ a maximal $F$-split torus in $\bG$ such that $\bM$ is the centralizer of $\bA$ in ${\bG}.$ 
Then, we have a maximal $F$-split torus $\bA^*$ and maximal $F$-torus $\bT^*$ of $\bG^*$ such that $\psi_0(\bA) \subset \bA^*,$ $\bT^* \subset \bM^*,$ and $\bT^*$ is the centralizer of $\bA^*$ in ${\bG^*}.$ From \cite[Lemma 8.2']{haines17}, we recall a finite morphism
\begin{equation} \label{intro fin morph}
\bar\psi_0: (Z(\widehat M)^I)_{\Phi}/W(G,A) \rightarrow (\widehat{T^*}^{I^*})_{\Phi^*}/W(G^*,A^*),
\end{equation}
where $\widehat M, \widehat{T^*}$ respectively denote the connected components of the $L$-groups of $\bM, \bT^*;$ $Z(\widehat M)$ is the center of $\widehat M;$
$I=Gal(\bar L /L)$ is the inertial group with the completion $L$ of the maximal unramified extension in $\bar F;$ 
$\bar L \supset \bar F$ is an algebraic closure of $L;$ $\Phi \in Gal(L/F)$ is the inverse of the Frobenius automorphism; $I^*$ and $\Phi^*$ indicate Galois actions on $\bG^*;$ and $W(\cdot,\cdot)$ is the relative Weyl group. The further detail can be found in Section \ref{weakunram}.

Given any unitary character $\chi \in X^{\rm{w}}(M),$ from the equality $X^{\rm{w}} (M) = (Z(\widehat{M})^I)_{\Phi}$ (see \eqref{equality}), we have a unitary character $\chi^* \in X^{\rm{w}}(T^*)$ such that  
\[
\chi^*=\bar\psi_0(\chi).
\]
We then have the following exact sequence (Theorem \ref{main for weakly}):
\begin{equation} \label{intro exseq 2}
1 \longrightarrow R_{\chi}   \longrightarrow  R_{\chi^*}  \longrightarrow  R_{\chi^*}/R_{\chi}   \longrightarrow 1,
\end{equation}
which yields the transfer of $R$-groups for unitary weakly unramified characters between a quasi-split group and its non-quasi-split $F$-inner forms.
Indeed, identifying $
W(G,A)=W(G,M)=W(\widehat G, \widehat M)=W(G^*,M^*)=W(\widehat{G^*}, \widehat{M^*}),$
we verify that the stabilizer $W(\chi)$ in $W(G,M)$ of $\chi$ is contained in the stabilizer $W(\chi^*)$ in $W(G^*, M^*)$ of $\chi^*,$ as  $\bar\psi_0$ is $W(G,A)$-invariant. 
Next, we reduce to the simply connected group and analyze two sets $\Delta'(\chi)$ and $\Delta'(\chi^*)$ of certain reduced roots (see \eqref{delta} for the definition) to reach the inclusion $\Delta'(\chi) \subset \Delta'(\chi^*).$ We then conclude that $W^\circ(\chi) \subset W^\circ(\chi^*),$ where $W^\circ(\bullet)$ denotes the subgroup generated by the reflections in the roots of $\Delta'(\bullet).$  
As by-product, the order of the quotient $R_{\chi^*}/R_{\chi}$ is 1,2,3, or 4, except $A_n$ in which case the order is a divisor of $n+1.$ This indicates the difference in numbers of conjugacy classes of good maximal compact (special maximal parahoric) subgroup of $\bG^*$ and $\bG$ (c.f., \cite[Section 3]{keys82unram} and Corollary \ref{kernel size}).
Further, it is expected that the quotient $R_{\chi^*}/R_{\chi}$ may be involved with the kernel of the finite morphism $\bar\psi_0$ (Remark \ref{rem for K}).

Finally, applying two exact sequences \eqref{intro exseq 1} and \eqref{intro exseq 2}, we obtain the classification  (Theorem \ref{final}) of the Knapp-Stein $R$-group $R_{\chi}$ in the general setting of any unitary weakly unramified character of a minimal $F$-Levi subgroup of any connected reductive group over $F$ whose derived group is simply-connected and almost simple. In the course of the proof, we observe a simple relationship \eqref{fullinclusion} among three Knapp-Stein $R$-groups for the given group, its quasi-split inner form, and its derived subgroup. 

We remark that this work generalizes Key's classification in Section \ref{unram rgp-keys} and describes the difference from those of unitary unramified representations (Remarks \ref{rem for commuting alg} and \ref{rem for last}). Furthermore, it transfers the $R$-groups for unitary weakly unramified characters between a quasi-split group and its non-quasi-split $F$-inner forms, which supports previous studies on behaviors of $R$-groups between inner forms \cite{chaoli, chgo12, cgclassic, cgsu, han04}.

This paper is organized as follows: In Section \ref{notation}, we give notation and background by recalling several definitions. Section \ref{secgentononsc} is devoted to revisiting Keys' classification and extending it to non-simply-connected cases. In Section \ref{rgps for weakunram}, after recalling notions of weakly unramified characters and finite morphisms, we transfer Knapp-Stein $R$-groups via finite morphisms between inner forms. This provides the structure of those $R$-groups for a general connected reductive group over a $p$-adic field.

\subsection*{Acknowledgements}
The author thanks the referee for a careful reading and valuable suggestions and comments that have improved the manuscript.
The author is supported by a grant from the Simons Foundation (\#840755).

\section{Notation and Background} \label{notation}

Let  $F$ be a finite extension of $\QQ_p.$ Denote by $\bar{F}$ an algebraic closure of $F.$
We denote by $\Gamma$ the absolute Galois group $\Gal(\bar{F} / F)$ and by $W_F$ the Weil group of $F.$ 
Given a connected reductive algebraic group $\bG$ defined over $F,$ the group $\bG(F)$ of $F$-rational points shall be written as $G$ and this is applied to other algebraic groups. Given a topological group $Y,$ we denote by $Z(Y)$ the center of $Y.$ Set $Y^D:=\Hom(Y , \CC^{\times})$ for the group of all continuous characters.

Let $\bG$ be a connected reductive algebraic group defined over $F.$  We fix a minimal $F$-parabolic subgroup $\bP_0$ of $\bG$ with Levi decomposition $\bP_0=\bM_0 \bN_0.$ 
By $\bA_0$ we denote the split component of $\bM_0$ and 
by $\Delta$ the set of simple roots of $\bA_0$ in $\bN_0.$ 
Given an $F$-parabolic subgroup  $\bP$ with Levi decomposition $\bP=\bM\bN$ such that $\bM \supseteq \bM_0$ and $\bN \subseteq \bN_0,$
we have a subset $\Theta \subseteq \Delta$ such that $\bM$ equals the Levi subgroup $\bM_{\Theta}$ generated by $\Theta.$
Set $\bA_{\bM_{\Theta}}=\bA_{\bM}$ for the split component of $\bM=\bM_{\Theta}.$  
Let $\Phi(P, A_M)$ denote the set of reduced roots of $\bP$ with respect to $\bA_\bM.$ 
By $W_M = W(\bG, \bA_\bM) := N_\bG(\bA_\bM) / Z_\bG(\bA_\bM)$ we denote the Weyl group of $\bA_\bM$ in $\bG.$ 
Here, $N_\bG(\bA_\bM)$ and $Z_\bG(\bA_\bM)$ are respectively the normalizer and centralizer of $\bA_\bM$ in $\bG.$ 
Fixing $\Gamma$-invariant splitting data, we define the $L$-group of $\bG$ as a semi-direct product $^{L}G := \widehat{G} \rtimes W_F.$ 
We shall refer the reader to \cite{bo79, borel91, springer98} for the detail. 

Denote by $\Irr(G)$ the set of isomorphism classes of irreducible admissible complex representations of $G.$ 
With no confusion, we shall not make a distinction between each isomorphism class and its representative. 
Given $\sigma \in \Irr(M),$ we write $i_M^G (\sigma)$ for the normalized (twisted by $\delta_{P}^{1/2}$) induced representation. Here, $\delta_P$ denotes the modulus character of $P.$ 
Denote by $\Pi_{\disc}(G)$  the subset of $\Irr(G)$ which consists of discrete series representations.
Here, a discrete series representation is an irreducible, admissible, unitary representation whose matrix coefficients are square-integrable modulo the center $Z(G)$ of $G.$

Given  $\sigma \in \Pi_{\disc}(M)$ and $w \in W_M,$  we consider the representation ${^w}\sigma$  given by ${^w}\sigma(x)=\sigma(w^{-1}xw).$  
It is noted that the isomorphism class of ${^w}\sigma$ is independent of the choices of representatives of $w \in W_M$ in $G.$  
Set the stabilizer $W(\sigma)$ of $\sigma$ in $W_M$ as follows:
\[ 
W(\sigma) = \{ w \in W_M : {^w}\sigma \s \sigma \}.
\]
Let
\begin{equation} \label{delta}
\Delta'(\sigma) =\{ \alpha \in \Phi(P, A_M) : \mu_{\alpha} (\sigma) = 0 \}.
\end{equation}
Here, $\mu_{\alpha} (\sigma)$ is the rank one Plancherel measure for $\sigma$ attached to $\alpha$ \cite[p.1108]{goldberg-class}. 
We define the \textit{Knapp-Stein $R$-group}:
\[
R_{\sigma} := \{ w \in W(\sigma) : w \alpha > 0, \; \forall \alpha \in \Delta'(\sigma) \}.
\]
Denote by $W^{\circ}_{\sigma}$ the subgroup of $W(\sigma)$ generated by the reflections in the roots of $\Delta'(\sigma).$
There is an isomorphism
\[
R_\sigma \s  W(\sigma)/W^{\circ}_{\sigma}.
\]
The detail can be referred to \cite{ks72, sil78, sil78cor}.

An 2-cocycle in $H^2(R_{\sigma}, \CC^{\times})$ yields a central extension $\widetilde{R_{\sigma}}$ of $R_{\sigma}$ by $\CC^\times,$
which implies the decomposition of $i_M^G(\sigma)$ into 
\[
\bigoplus_{\rho \in \Pi_{-}(\widetilde{R_{\sigma}})} \rho \boxtimes \pi_\rho
\]
as representations of $\widetilde{R_{\sigma}} \times G.$
Here, 
\[
\Pi_{-}(\widetilde{R_{\sigma}}) = \{ \mbox{irreducible representation $\rho$ of}~ \widetilde{R_{\sigma}} : \rho(z)=z \cdot id \mbox{ for all } z \in \CC^{\times} \},
\]
and $\rho \mapsto \pi_\rho$ is a bijection from $\Pi_{-}(\widetilde{R_{\sigma}})$ to the set $\J(\sigma)$ of inequivalent irreducible constituents of $i_M^G(\sigma).$
The detail can be referred to \cite[Section 2]{art93} and \cite[Section 2.3]{chaoli}.

We note that the group algebra $\CC[R_{\sigma}]_{\eta}$ of $R_{\sigma}$ twisted by a $2$-cocycle $\eta$ in $H^2(R_{\sigma}, \CC^{\times})$ is $\CC$-algebra isomorphic to the commuting algebra ${\End}_{G}(i_M^G (\sigma))$ of $G$-endomorphisms of $i_M^G (\sigma)$ (c.f., From \cite{sil78, sil78cor, art93}).  
In particular, when $\sigma$ is a unitary character of $M$ with a minimal $F$-Levi subgroup $\bM,$ it is true that $H^2(R_\sigma, \CC^{\times})=1$ so that we can identify $\Pi_{-}(\widetilde{R_{\sigma}})$ with the set of irreducible representations of $R_{\sigma}$ (see \cite[Theorem 2.4]{keys87}).

\section{Generalization to non-simply-connected cases} \label{secgentononsc}

This section is devoted to generalizing Keys' work towards unitary unramified characters of $F$-minimal Levi subgroups of a general connected reductive group (Theorem \ref{thm for general}).

\subsection{Revisiting Knapp-Stein $R$-groups for simply-connected, almost simple, semi-simple groups} \label{unram rgp-keys}

We recall Keys' work in \cite{keys82unram} for Knapp-Stein $R$-groups for simply-connected, almost simple, semi-simple groups. Let $\bold{H}$ be a simply-connected, almost simple, semi-simple group over $F.$ Let $\bM_H$ be a minimal $F$-Levi subgroup of $\bold{H}.$ Given a unitary unramified character $\chi_{M_H}$ of $M_H,$ from \cite[Section 3]{keys82unram}, each non-trivial Knapp-Stein $R$-group $R_{\chi_{M_H}}$ for $\chi_{M_H}$ has the following structure:
\[ 
R_{\chi_{M_H}} \s 
\left\{ 
\begin{array}{lllll}
       \ZZ/d\ZZ, \mbox{where $n+1$ is divisible by $d$}, & 
         \mbox{if $\bold{H}$ is of $A_n$-type};\\
         \ZZ/2\ZZ, &  
         \mbox{if $\bold{H}$ is of $B_n, C_n,$ or $E_7$-type};\\ 
        \ZZ/2\ZZ, \ZZ/2\ZZ \times \ZZ/2\ZZ, \mbox{ or } \ZZ/4\ZZ, &  
         \mbox{if $\bold{H}$ is of $D_n$-type};\\ 
        \ZZ/3\ZZ, &  
         \mbox{if $\bold{H}$ is of $E_6$-type}.
\end{array} \right. 
\]
\subsection{Generalizing $R$-groups towards non-simply-connected case}  \label{gen to non-sc}
Let $\bG$ be a connected reductive group over a $p$-adic field $F$ such that its derived subgroup $\bG_{\der}$ is simply-connected and almost simple. There are many cases satisfying this condition such as $\bG=\GL_n, \GSp_{2n}, \GSpin_{n}, \GU_n, \U_n,$ etc.
Let $\bM$ be an $F$-Levi subgroup of $\bG,$ and let $\bM_{\flat}$ be an $F$-Levi subgroup of $\bG_{\der}$ such that $\bM_\flat = \bM \cap \bG_{\der}.$

Recalling \cite{gk82} and \cite{tad92}, given $\sigma \in \Irr(M),$ we consider the restriction ${\Res}_{M_\flat}^{M}(\sigma)$ of $\sigma$ from $M$ to $M_\flat,$ and produce the following finite set
\begin{equation} \label{restset}
\Pi_{\sigma}(M):=\{\tau \hookrightarrow {\Res}_{M_\flat}^{M}(\sigma) \} /\s
\end{equation}
consisting of equivalence classes of all irreducible constituents of the restriction ${\Res}_{M_\flat}^{M}(\sigma).$ It is well-known that
for any $\sigma_1$, $\sigma_2 \in \Irr(M),$ the intersection $\Pi_{\sigma_1}(M) \cap \Pi_{\sigma_2}(M)$ is non-empty if and only if
$\sigma_1 \s \sigma_2 \eta$ for some $\eta \in (M / M_\flat)^D$ 
if and only if the equality $\Pi_{\sigma_1}(M) = \Pi_{\sigma_2}(M)$ holds.  We also have the following decomposition
\begin{equation*} \label{decomp of Res}
{\Res}_{M_\flat}^{M}(\sigma) = \bigoplus _{\tau \in \Pi_{\sigma}(M_\flat)} \langle \tau, \sigma \rangle_{M_\flat} \cdot \tau,
\end{equation*}
where $\langle \tau, \sigma \rangle_{M_\flat}:= \dim_{\CC}
\; {\Hom}_{M_\flat}(\tau, {\Res}_{M_\flat}^{M} (\sigma)).$
For any $\tau_1,~ \tau_2 \in \Pi_{\sigma}(M_\flat),$ we have
$\langle \tau_1, \sigma \rangle_{M_\flat} = \langle \tau_2, \sigma \rangle_{M_\flat}.$
Define
\begin{equation*} \label{X(sigma)}
X(\sigma):= \{ \eta \in (M/M_\flat)^D : \sigma \s \sigma \eta \},
\end{equation*}
and then have
\begin{equation} \label{cardinality}
|X(\sigma)| 
= 
{\dim}_{\CC} \; {\Hom}_{M_\flat}({\Res}_{M_\flat}^{M} (\sigma), {\Res}_{M_\flat}^{M} (\sigma))
= 
|\Pi_{\sigma}(M_\flat)| \cdot \langle \tau, \sigma \rangle_{M_\flat}^2
\end{equation}
(see \cite[Lemma 2.1(d)]{gk82}, \cite[Proposition 2.4]{tad92}, and \cite[Proposition 3.1]{choiymulti}).
\begin{rem}
One notes that those results recalled from \cite{gk82} and \cite{tad92} are  available for any connected reductive algebraic group $\bG_1$ over $F$ and its closed subgroup $\bG_2$ having the same derived subgroup, so that the quotient $G_1/G_2Z(G_1)$ of $F$-rational points is a finite abelian group (c.f., \cite[Section 2]{tad92} and \cite[Section 3]{chaoli}).
In addition, the transition between a connected reductive algebraic group over $F$ and its $F$-rational points can be done due the fact that the first Galois cohomology of the derived subgroup is always a finite abelian as it has the structure of the Pontryagin dual of the connected components of the center of its connected Langlands dual group (c.f., \cite[Section 6]{kot84}).
\end{rem}

We now specialize in the above setting when $\bM$ is a minimal $F$-Levi subgroup and so is $\bM_\flat.$
Let $\chi$ be a unitary character of $M.$ We consider the restriction ${\Res}_{M_\flat}^{M}(\chi)$
and obtain a unitary character $\chi_{\flat}$ of $M_{\flat}$ such that 
$\chi_\flat \hookrightarrow {\Res}_{M_\flat}^{M}(\chi).$
Since $X(\chi)=1,$ due to \eqref{cardinality}, the set $\Pi_{\chi}(M_\flat)$ is a singleton and $\langle \chi_\flat, \chi \rangle_{M_\flat}=1.$ Then, the inclusion becomes equality: 
\[
\chi_\flat ={\Res}_{M_\flat}^{M}(\chi).
\]
Following \cite[Definition 3.1]{go06}, we define
\begin{equation} \label{hatW}
\widehat{W(\chi)} :=\{\eta \in (M/M_\flat)^D : {^w}\chi \s \chi \eta \text{ for some } w \in W(\chi_\flat)  \},
\end{equation}
which is a finite abelian group.
We then have the following relationship between $R_\chi, R_{\chi_\flat},$ and  $\widehat{W(\chi)}.$
\begin{thm} \label{thm for general}
Let $\bM$ and $\bM_{\flat}$ be minimal $F$-Levi subgroups of $\bG$ and $\bG_{\der},$ respectively, as above. Given any unitary character $\chi$ of $M$ and $\chi_\flat = {\Res}_{M_\flat}^{M}(\chi),$ we have:
\begin{equation} \label{intro exact for gspin 2n}
	1 \longrightarrow R_{\chi}  \longrightarrow R_{\chi_\flat} \longrightarrow \widehat{W(\chi)} 
	\longrightarrow 1.
\end{equation}
\end{thm}
\begin{proof}
Since the Plancherel measure is compatible with restriction (c.f., \cite[Proposition 2.4]{choiy1}, \cite[Lemma 2.3]{go06}) and $\Phi(P_\flat, A_{M_\flat}) = \Phi(P, A_{M}),$ 
we have 
\[
W^{\circ}_{\chi} = W^{\circ}_{\chi_\flat}.
\] 
As ${\Res}_{M_\flat}^{M}(\chi)$ itself is $\chi_\flat,$ the isomorphism ${^w}\chi \s \chi$ implies ${^w}\chi_\flat \s \chi_\flat.$ This gives an embedding \[
R_\chi  \hookrightarrow R_{\chi_\flat}.
\]
On the other hand, given ${^w}\chi_\flat \s \chi_\flat,$ due to the argument below \eqref{restset}, there is $\eta \in (M/M_\flat)^D$ such that ${^w}\chi \s \chi \eta.$ 
We thus have the cokernel and then the exact sequence \eqref{intro exact for gspin 2n}. This completes the proof.
\end{proof}

\begin{rem}
By taking a unitary unramifined character $\chi$ of $M,$ Theorem \ref{thm for general} extends the Keys' work for unitary unramified cases from simply-connected groups to general groups. To relate $\chi$ and $\chi_\flat,$ one may carry out the opposite direction, i.e., lifting a given $\chi_\flat$ of $M_\flat$ to $M,$ since both representations are characters of minimal Levi subgroups and the quotient $M/M_{\flat}Z(M)$ is a finite abelian group. We also note that a similar method in the proof of Theorem \ref{thm for general} has been used in \cite[Proposition 3.2]{go06} for $\U_n$ and $\SU_n,$ and \cite[Sections 6 and 8]{bcg21} for $\GSpin_n$ and $\Spin_n.$
\end{rem}

\section{$R$-groups for unitary weakly unramified characters in general via finite morphisms} \label{rgps for weakunram}

This section is devoted to transferring Knapp-Stein $R$-groups for unitary weakly unramified characters of $F$-minimal Levi subgroups between a $p$-adic quasi-split group and its non-quasi-split inner forms (Theorem \ref{main for weakly}), and providing the structure of those $R$-groups for a general connected reductive group over a $p$-adic field (Theorem \ref{final}).
Therefore, Keys' work will be finally extended to unitary weakly unramified characters in a general setting.

\subsection{Weakly unramified representations and finite morphisms} \label{weakunram}
We recall notions of weakly unramified characters and finite morphisms. We mainly refer to \cite{haines14, haines15, haines17}.

\subsubsection{Weakly unramified representations} \label{subsub wurep} 
Let $\bG$ be a connected reductive algebraic group over $F.$ Let $\bM$ be an $F$-Levi subgroup of $\bG.$
We set 
\[
M_1:= \ker (\kappa_M) \cap M,
\]
where $\kappa_M: M \twoheadrightarrow X^*(Z(\widehat{M})^{\Phi}_I$
is the $\Phi-$fixed points of Kottwitz's functorial surjective homomorphism 
$\kappa_M: \bM(L) \twoheadrightarrow X^*(Z(\widehat{M})_I$ \cite[Section 7]{kot97}. Here, $L$ is the completion of the maximal unramified extension in $\bar F$ of $F,$  $\bar L \supset \bar F$ is an algebraic closure of $L,$ $I=Gal(\bar L /L)$ is the inertial group, $\Phi \in Gal(L/F)$ is the inverse of the Frobenius automorphism, and $X^*(Z(\widehat{M})_I$ is the group of coinvariants of $I$ in the character group $X^*(Z(\widehat{M}).$ 

Following \cite[Section 3.3]{haines14} and \cite[Section 3]{haines15}, we define the group of weakly unramified characters on $M$ as follows:
\[
X^{\rm{w}} (M):= (M/M_1)^D = {\Hom}(M/M_1, \CC^\times).
\]
We note that $X^{\rm{w}} (M)$ is a diagonal group over $\CC.$  By Kottwitz, there is an isomorphism
\[
M/M_1 \s X^*((Z(\widehat{M})^I)_{\Phi}),
\]
which implies that
\begin{equation} \label{equality}
X^{\rm{w}} (M) = (Z(\widehat{M})^I)_{\Phi}.
\end{equation}

\begin{rem} \label{good rem}
From \cite{haines14, hai10}, we note that
\begin{itemize}
\item Denote $X(M)$ by the group of unramified characters on $M.$ Then, the group $X(M)$ is the identity component of $X^{\rm{w}}(M)$ \cite[Section 3]{haines14}. 
\item When $\bG$ is simply-connected and semi-simple or $\bG$ is unramified over $F,$ then we have 
\[
X^{\rm{w}}(M)=X(M)
\]
due to \cite[Lemma 11.3.4]{haines14} and \cite[Sections 1 and 11]{hai10}. 

\item Any unitary $\chi \in X^{\rm{w}}(M)$ with  a minimal $F$-Levi subgroup $\bM$ is a discrete series representation of $M.$ 
\end{itemize}
\end{rem}

Let $K$ be a special maximal parahoric subgroup in $G.$ Choose any maximal $F$-split torus $\bA$ whose associated apartment in the Bruhat-Tits building contains the special vertex associated to $K.$ 
Let $\bM$ be a minimal $F$-Levi subgroup such that $\bM=Z_{\bG}(\bA).$ Denote by $W(G,A)$ the relative Weyl group. Let $\Pi(\bG/F, K)$ denote the set of isomorphism classes of smooth irreducible representations $\pi$ of $G$ such that $\pi^K \neq 0.$ Then we have a 1-1 bijection (\cite[Proposition 3.1]{haines15}):
\[
(Z(\widehat{M})^I)_{\Phi} / W(G,A) \overset{\s}{\longrightarrow} \Pi(\bG/F, K),
\]
which maps $\chi \in X^{\rm{w}} (M)$ to $\pi_\chi \in  \Pi(\bG/F, K).$ Here, 
$\pi_\chi$ is the smallest $G$-stable subspace of the right regular representation of locally constant functions on $G$ containing a certain spherical function defined by $\chi$ and $K.$ 

\subsubsection{Finite morphisms} \label{subsub finitemorph} 
Let $\bG^*$ be a quasi-split connected reductive group over $F.$
Let $(\bG, \Psi)$ be an $F$-inner form of $\bG^*$ so that $\Psi$ is a $\Gamma$-stable $\bG^*_{\ad}(\bar F)$-orbit of $\bar F$-isomorphisms $\psi: \bG \rightarrow \bG^*$ (c.f., \cite{shin10, bo79}). Following \cite[Section 11]{haines14} and \cite[Section 8]{haines15}, we fix once and for all parahoric subgroups $J \subset G$ and $J^* \subset G^*.$ Choose any maximal $F$-split tori $\bA$ in $\bG$ and $\bA^*$ in $\bG^*$ such that the facets fixed by $J$ and $J^*$ is respectively contained in the apartments of the Bruhat-Tits buildings of $\bG$ and $\bG^*$ corresponding to the torus $\bA$ and $\bA^*.$ 
Let $\bM = Z_{\bG}(\bA)$ and $\bT^* =Z_{\bG^*}(\bA^*).$ Then $\bM$ is a minimal $F$-Levi subgroup of $\bG$ and $\bT^*$ is a maximal $F$-torus in $\bG^*.$
Let $\bP$ in $\bG$ be an $F$-parabolic subgroup with Levi factor $\bM.$
Let $\bB^*$ in $\bG^*$ be a Borel subgroup with Levi factor $\bT^*.$
Then there exists a unique $F$-parabolic subgroup $\bP^* \supset \bB^*$ in $\bG^*$ such that $\bP^*$ is $\bG^*(\bar F)$-conjugate to $\psi(\bP)$ for every $\psi \in \Psi.$ Let $\bM^*$ be the unique Levi factor of $\bP^*$ containing $\bT^*.$ It is clear that $\bM$ is an $F$-inner form of $\bM^*.$ Then, there is $\phi_0 \in \Psi$ such that $\phi_0(\bP)=\bP^*, \phi_0(\bM)=\bM^*,$ and $\phi_0(\bA) \subset \bA^*.$

Recalling the notation $I, \Phi$ from Section \ref{subsub wurep}, there is a finite morphism (\cite[Lemma 8.2']{haines17})
\begin{equation} \label{fin morph}
\bar\psi_0: (Z(\widehat M)^I)_{\Phi}/W(G,A) \rightarrow (\widehat{T^*}^{I^*})_{\Phi^*}/W(G^*,A^*).
\end{equation}
Here, $I^*$ and $\Phi^*$ indicate Galois actions on $\bG^*$ and $W(\cdot,\cdot)$ is the relative Weyl group.
\begin{rem}
From \cite[Section 2]{haines17}, we note that:
\begin{itemize}
\item $\bar\psi_0$ is not necessarily an immersion in general, since the injective map 
\[
W(G,A) \rightarrow W(G^*,M^*)/W(G^*,A^*)
\]
is not necessarily surjective (see \cite[(8.5)]{haines15} and \cite[Section 2.4]{haines17});
\item  $\bar\psi_0$ is an immersion when $\bG^*=\GL_n$ (see \cite[Lemma 2.6]{haines17});
\item $\bar\psi_0$ is isomorphic when $\bG \s \bG^*$ over $F.$
\end{itemize}
\end{rem}

It then follows from \cite[Sections 9 and 10]{haines15} that
\begin{equation} \label{satake}
(\widehat{T^*}^{I^*})_{\Phi^*}/W(G^*,A^*) \overset{\s}{\rightarrow} [\widehat{G^*}^{I^*} \rtimes \Phi^*]_{ss}/\widehat{G^*}^{I^*} = [\widehat{G}^{I} \rtimes \Phi]_{ss}/\widehat{G}^{I},
\end{equation}
where $[\bullet \rtimes \bullet]_{ss}$ denotes the set of semi-simple elements in the coset $\bullet \rtimes \bullet.$
\subsection{Transfer of $R$-groups via finite morphisms}
We shall continue with notation in Section \ref{subsub finitemorph}.
Let $\chi \in X^{\rm{w}}(M)$ be given. Then we have $\chi^* \in X^{\rm{w}}(T^*)$ such that  
\[
\chi^*=\bar\psi_0(\chi)
\]
via the finite morphism $\bar\psi_0$ in \eqref{fin morph}. Note that if $\chi$ is unitary then so is $\chi^*.$

\begin{thm} \label{main for weakly}
With notation above, for any unitary $\chi \in X^{\rm{w}}(M)$ and $\chi^* \in X^{\rm{w}}(T^*)$ such that  $\chi^*=\bar\psi_0(\chi),$ we have the following exact sequence
\[
1 \longrightarrow R_{\chi}   \longrightarrow  R_{\chi^*}  \longrightarrow  R_{\chi^*}/R_{\chi}   \longrightarrow 1.
\]
\end{thm}
\begin{proof}
Let $w \in W(\chi)$ be given. Note that since $\widehat M=\widehat{M^*}$ we have the canonical equality
\begin{equation} \label{identification 1}
X^*((Z(\widehat{M})^I)_{\Phi}) = X^*((Z(\widehat{M^*})^{I^*})_{\Phi^*}).
\end{equation}
We consider $\chi^* \in X^*((Z(\widehat{M^*})^{I^*})_{\Phi^*})$ (c.f., \cite[Remark 11.12.2]{haines14}). 
We identify
\begin{equation} \label{identification 2}
W(G,A)=W(G,M)=W(\widehat G, \widehat M)=W(G^*,M^*)=W(\widehat{G^*}, \widehat{M^*}).
\end{equation}
Thus, since $w \in W(G,M)$ and ${^w}\chi \s  \chi,$ and since $\bar\psi_0$ is $W(G,A)$-invariant due to \cite[Lemma 11.12.4]{haines14}, we have ${^w}\chi^* \s  \chi^*,$
which implies 
\[
W(\chi) \subset W(\chi^*).
\]

As in the proof of Theorem \ref{thm for general}, since the Plancherel measure is compatible with restriction (c.f., \cite[Proposition 2.4]{choiy1}, \cite[Lemma 2.3]{go06}), we reduce to the simply connected group. 
Applying the description from \cite[Section 3]{keys82unram} and \cite[Chapter V]{mac71} to $\chi$ and $\chi^*,$ we have
\[
\Delta'(\chi)=\{ \alpha \in \Omega: \chi(t_\alpha)=1 \text{ if } q_{\alpha/2}=1, \text{ and } \chi(t_\alpha)=\pm 1 \text{ if } q_{\alpha/2}\neq 1  \}
\]
and 
\[
\Delta'(\chi^*)=\{ \alpha^* \in \Omega^*: \chi^*(t_{\alpha^*})=1 \text{ if } q_{\alpha^*/2}=1, \text{ and } \chi^*(t_{\alpha^*})=\pm 1 \text{ if } q_{\alpha^*/2}\neq 1  \}.
\]
Here, $\Omega$ is a reduced irreducible root system (c.f., \cite[Section 1]{keys82unram}); $t_{\alpha}$ is the translation $w_\alpha \circ w_{\alpha+1}$ with reflection $w_\alpha$ in the hyperplanes for $\alpha;$ $q_{\alpha/2}=\frac{[U_\alpha:U_{\alpha+1}]}{[U_{\alpha-1}:U_{\alpha}]};$ 
and $U_{\alpha+n}=u_{\alpha}(\mathfrak{P}^n)$ is the bounded subgroup of the root group $u_\alpha(F^{un})$ corresponding to $\alpha$ with the maximal unramified extension $F^{un}$ of $F$ and prime ideal $\mathfrak{P}.$ We refer to \cite[Section 1]{keys82unram} and \cite[Section 2]{ree08}.
Since $\Omega \subset \Omega^*$ and $q_{\alpha/2}=q_{\alpha^*/2},$ we have
\[
\Delta'(\chi) \subset \Delta'(\chi^*),
\]
which implies
\[
W^\circ(\chi) \subset W^\circ(\chi^*).
\]
Note that any $w \in W(\chi) \cap W^\circ(\chi^*)$ lies in $W(G,M).$ Since $\bar\psi_0$ is $W(G,A)$-invariant, it follows from \eqref{identification 1} that for $\alpha^*=\alpha \in W(G^*,M^*)=W(G,M)$ with the identification \eqref{identification 2}, $\chi^*(t_{\alpha^*}) = 1$ (respectively, $\pm 1$) implies $\chi(t_\alpha) = 1$ (respectively, $\pm 1$).  We indeed have 
\[
W(\chi) \cap W^\circ(\chi^*) = W^\circ(\chi).
\]
Therefore, as $R_\chi$ and $R_{\chi^*}$ are abelian due to Theorem \ref{thm for general}, we complete the proof.
\end{proof}

\begin{rem}
Theorem \ref{main for weakly} transfers Knapp-Stein $R$-groups for unitary weakly unramified characters between a $p$-adic quasi-split group and its non-quasi-split inner forms.
\end{rem}

\begin{cor} \label{kernel size}
With above notation, the quotient $R_{\chi^*}/R_{\chi}$ is of the form:
\[
 \left\{ 
\begin{array}{l l l l l}
   \ZZ/d\ZZ \text{ with $d|(n+1),$ } & \text{if } \bG \text{ is of $A_n$-type} \\
   1 \text{ or } \ZZ/2\ZZ, & \text{if } \bG \text{ is of $B_n, C_n, E_7$-type} \\
1,  \ZZ/2\ZZ, \ZZ/2\ZZ \times \ZZ/2\ZZ, \text{ or } \ZZ/4\ZZ, & \text{if } \bG \text{ is of $D_n$-type} \\
1 \text{ or }  \ZZ/3\ZZ,& \text{if }  \bG \text{ is of $E_6$-type} \\
 1, & \text{if } \bG \text{ is of $E_8, F_4, G_2$-type.} 
  \end{array}
  \right.
\]
\end{cor}
\begin{proof}
Let $\chi$ be a unitary weakly unramified character in $X^{\rm{w}}(M).$ Via the finite morphism, we have its image $\chi^*$ in $X^{\rm{w}}(M^*).$ Due to Theorem \ref{thm for general}, we have the structure of $R_{\chi^*}$ and  then Theorem \ref{main for weakly} yields the quotient $R_{\chi}/R_{\chi^*}.$ This completes the proof.

\end{proof}

\begin{rem} \label{rem for K}
As argued in \cite[Section 3]{keys82unram}, the difference in numbers of conjugacy classes of good maximal compact (special maximal parahoric) subgroup of $\bG^*$ and $\bG$ is exactly the order of the quotient $R_{\chi^*}/R_{\chi}.$
Further, the order of the quotient $R_{\chi^*}/R_{\chi}$ may be related to the difference between $\tilde K$ and $K$ in \cite[Sections 1 and 11]{hai10}, and further involved with the kernel of the finite morphism $\bar\psi_0$ in \eqref{fin morph}.
\end{rem}

Lastly, we give the following classification of the Knapp-Stein $R$-group for any unitary weakly unramified character in a general setting.

\begin{thm} \label{final}
Let $\bG$ be a connected reductive group over $F.$ Assume that $\bG_{\der}$ is simply-connected and almost simple. Let $\bM$ be a minimal $F$-Levi subgroup of $\bG.$ Given any unitary weakly unramified character $\chi$ in $X^{\rm{w}}(M),$ we have
\begin{itemize}
\item If $\bG$ is of $A_n$-type, then $R_{\chi}$ is isomorphic to a subgroup of  $\ZZ/d\ZZ$ where $n+1$ is divisible by $d;$
\item If $\bG$ is of $B_n$-type, then $R_{\chi}$ is trivial or isomorphic to $\ZZ/2\ZZ;$
\item If $\bG$ is of $C_n$-type, then $R_{\chi}$ is trivial or isomorphic to $\ZZ/2\ZZ;$
\item If $\bG$ is of $D_n$-type, then $R_{\chi}$ is trivial or isomorphic to $\ZZ/2\ZZ,$ $\ZZ/2\ZZ \times \ZZ/2\ZZ,$ or $\ZZ/4\ZZ;$
\item If $\bG$ is of $E_6$-type, then $R_{\chi}$ is trivial or isomorphic to $\ZZ/3\ZZ;$
\item If $\bG$ is of $E_7$-type, then $R_{\chi}$ is trivial or isomorphic to $\ZZ/2\ZZ;$
\item If $\bG$ is of $E_8, F_4, G_2$-type, then $R_{\chi}$ is trivial.
\end{itemize}
\end{thm}
\begin{proof}
Let $\chi$ be a unitary weakly unramified character of $M.$ Let $\bG^*$ be a quasi-split inner form of $\bG.$ Set $\chi^* \in X^{\rm{w}}(T^*)$ such that  $\chi^*=\bar\psi_0(\chi).$ 
Let $\bT^*_\flat$ be a minimal $F$-Levi subgroup of $\bG^*_{\der}.$ Since $\bG^*_{\der}$ is simply-connected and almost simple, we have $X^{\rm{w}} (T^*_\flat)=X (T^*_\flat)$ (see Remark \ref{good rem}). 
Denote $\chi^*_\flat \in X (T^*_\flat)$ such that 
\[
\chi^*_\flat \hookrightarrow {\Res}_{T^*_\flat}^{T^*}(\chi^*).
\]
Applying Theorems \ref{thm for general} and \ref{main for weakly}, we have 
\begin{equation} \label{fullinclusion}
R_{\chi} \subset R_{\chi^*} \subset R_{\chi^*_\flat}.
\end{equation} 
Due to Keys' result in Section \ref{unram rgp-keys}, the proof is complete.
\end{proof}

\begin{rem} \label{rem for commuting alg}
Recalling notation in Section \ref{notation},  it follows from Theorem \ref{final} that we have an $\CC$-algebra isomorphism ${\End}_{G}(i_M^G (\chi)) \s \CC[R_\chi],$ as an 2-cocycle in $H^2(R_{\chi^*}, \CC^{\times})$ is trivial due to \cite[Theorem 2.4]{keys87}.  
\end{rem}
\begin{rem} \label{rem for last}
Theorem \ref{final} generalizes Key's result in Section \ref{unram rgp-keys} to unitary weakly unramified characters and then it describes the difference from those $R$-groups for unitary unramified characters with help of Theorem \ref{thm for general}.
\end{rem}

\end{document}